\documentclass[12pt]{amsart}

\usepackage{amssymb}
\usepackage{graphicx}
\usepackage{enumerate}
\usepackage{multirow}
\usepackage{amsmath,color}
\usepackage{hyperref}
\usepackage{url}
\usepackage[colorinlistoftodos]{todonotes}
\usepackage{comment}

\newtheorem{thm}{Theorem}[section]

\newtheorem{lem}[thm]{Lemma}
\newtheorem{prop}[thm]{Proposition}

\numberwithin{equation}{section}
\theoremstyle{definition}

\title{Almost amorphic association schemes}
\author{Edwin R. van Dam}
\address{Department of Econometrics and O.R., Tilburg University, The Netherlands}
\email{Edwin.vanDam@uvt.nl}
\author{Jack H. Koolen}
\address{School of Mathematical Sciences, University of Science and Technology of China, Hefei, Anhui, PR China}
\address{Wen-Tsun Wu Key Laboratory of CAS, Hefei,Anhui, PR China}
\email{koolen@ustc.edu.cn}
\author{Yanzhen Xiong}
\address{College of Science, National University of Defense Technology, Changsha, Hunan, PR China}
\email{xiongyanzhen@nudt.edu.cn}
\date{\today} 

\begin{document}

\begin{abstract} An association scheme is called amorphic if every possible fusion of relations gives rise to another association scheme. In earlier work, we showed that if an association scheme has at most one relation that is neither strongly regular of Latin square type nor strongly regular of negative Latin square type, then it must be amorphic. We now construct non-amorphic $d$-class association schemes in which precisely two relations are not strongly regular of Latin square type or strongly regular of negative Latin square type, for any $d \geq 4$. We also raise the question whether different types of strongly regular graphs can coexist in an association scheme. Among some other results, we show that if one of the relations is a lattice graph, then any other strongly regular relation in the scheme must be of Latin square type.
\end{abstract}

\maketitle

\section{Introduction}
An association scheme is called amorphic if every possible fusion of relations gives rise to another association scheme. In earlier work \cite{DKX}, we showed that if an association has at most one relation that is neither strongly regular of Latin square type nor strongly regular of negative Latin square type, then it must be amorphic. We now construct non-amorphic $d$-class association schemes in which precisely two relations are not strongly regular of Latin square type or strongly regular of negative Latin square type, for any $d \geq 4$.

We also raise the question whether different types of strongly regular graphs can coexist in an association scheme. Among some other results, we show that if one of the relations is a lattice graph, then any other strongly regular relation in the scheme must be of Latin square type.

This paper is further organized as follows. After the preliminaries, we start with a key example in Section \ref{sec:sylvester} that is obtained from the Sylvester graph. 
In Section \ref{sec:combining}, we develop our main tool in Theorem \ref{thm:main} for possible constructions of association schemes with many Latin square type relations. In Section \ref{sec:constructions}, we obtain our main result in Theorem \ref{thm:mainresult} by fissioning relations in the association scheme of the Brouwer-Pasechnik graphs.
In Section \ref{sec:NLS}, we focus on association schemes with strongly regular graphs of negative Latin square type. In the final Section \ref{sec:types}, we raise the question of whether different types of strongly regular graphs can coexist in an association scheme. 

\section{Preliminaries}

\subsection{Association schemes}
A $d$-class (symmetric\footnote{All association schemes in this paper are symmetric}) {\em association scheme} $\mathbb{A}$ is a set of $d+1$ symmetric $01$-matrices $\{A_0=I,A_1,\dots,A_d\}$ of the same size $v$ that sum to the all-ones matrix $J$ and whose span forms an algebra --- the {\em Bose-Mesner algebra} --- with respect to ordinary multiplication.

The basis matrices $A_i, i =0,1,\dots,d$ are the adjacency matrices of the so-called {\em relations} of the scheme, but we will typically interpret them as graphs (except $A_0=I$).

As the Bose-Mesner algebra is closed under multiplication, there are so-called {\em intersection numbers} $p_{ij}^h$ satisfying
$A_iA_j = \sum_{h=0}^d p_{ij}^{h}A_h$ for all $h,i,j=0,1,\ldots,d$.

Because the basis matrices are symmetric, it follows that the Bose-Mesner algebra is commutative. It also follows that it has a basis of (minimal) {\em idempotents} $\{E_j : j=0,1,\ldots,d\}$ consisting of projection matrices onto the common eigenspaces of the basis matrices. One of these matrices is $E_0=\frac1v J$. Thus, there are {\em eigenvalues} $P_{ji}$ such that 
$$A_iE_j=P_{ji}E_j$$ for all $i,j=0,1,\ldots,d$.
Because the Bose-Mesner algebra is also closed under entrywise multiplication $\circ$ (and the matrices $A_i, i=0,1,\dots,d$ are $\circ$-idempotent), there are also {\em dual eigenvalues} $Q_{ij}$ such that
$$vE_j \circ A_i=Q_{ij}A_i$$ for all $i,j=0,1,\ldots,d$. The matrices $P$ and $Q$ are called the (first and second) {\em eigenmatrices} of the association scheme.

In particular, we have that $P_{j0}=1$ for all $j$ and $P_{0i}=p_{ii}^0=:k_i$, the {\em valency} of relation $A_i$. The other eigenvalues $P_{ji}, j \neq 0$ are called the {\em restricted eigenvalues} of $A_i$ ($i=0,1,\ldots,d$).

Dually, we have that $Q_{i0}=1$ for all $i$ and $Q_{0j}=:m_j$, the {\em rank} of idempotent $E_j$
($j=0,1,\ldots,d$). 

If $\{I,B_1,...B_e\}$ is an association scheme $\mathbb{B}$ in which each of the $B_i$ is a sum of matrices in the association scheme $\mathbb{A}=\{I,A_1,\dots,A_d\}$, then we say that $\mathbb{B}$ is a {\em fusion} scheme of $\mathbb{A}$, and $\mathbb{A}$ is a {\em fission} scheme of $\mathbb{B}$. If every possible way of taking sums of matrices in $\mathbb{A}$ gives rise to a fusion scheme, then we call $\mathbb{B}$ {\em amorphic}.

For more basic properties of association schemes, we refer to \cite{BBIT}. 

\subsection{Strongly regular graphs}
A graph is called strongly regular if it is one of the relations in a $2$-class association scheme. For background on strongly regular graphs, we refer to \cite{BvMSRG}.

A strongly regular graph with $v$ vertices and valency $k$ is of \emph{Latin square type} if $v = n^2$, $k = t(n-1)$, and the restricted eigenvalues are $n-t$, and $-t$, for some positive integers $n$ and $t$. It is of {\em negative Latin square type} if $n$ and $t$ are negative integers. A {\em conference graph} is a strongly regular graph with $k=\frac{1}{2}(v-1)$ and restricted eigenvalues $-\frac12 \pm \frac12 \sqrt{v}$. The conference graphs on a square number of vertices $v$ are precisely the graphs that are both of Latin square type and of negative Latin square type. We say that a strongly regular graph is of \emph{strictly} (negative) Latin square type if it is of (negative) Latin square type but it is not a conference graph. 

In an amorphic $d$-class association scheme with $d \ge 3$, all relations are strongly regular of Latin square type or all are strongly regular of negative Latin square type. For background on amorphic association schemes, we refer to \cite{vDM2010}.

Finally, we will use the following two lemmas.

\begin{lem}\cite[Lemma 2.4]{DKX}\label{lem:SRGkandaNLS}
    Let $\Gamma$ be a strongly regular graph on $n^2$ vertices and valency $k$, having a restricted eigenvalue $a$, such that $k=-a(n-1)$. Then $\Gamma$ is of  Latin square type if $n$ is positive and $\Gamma$ is of negative Latin square type if $n$ is negative.
\end{lem}

A {\em square spread} in a graph on $n^2$ vertices is a subgraph that is a disjoint union of $n$ cliques of order $n$.  

\begin{lem}\label{lem:spread}
Let $\Gamma$ be a strongly regular graph of Latin square type on $n^2$ vertices that contains a square spread. Then the graph that is obtained by removing the edges of the square spread from $\Gamma$ is strongly regular of Latin square type.
\end{lem}

\begin{proof}
    This follows from \cite{HT}.
\end{proof}

\section{The Sylvester graph}\label{sec:sylvester}
Recall that we are interested in constructing $d$-class association schemes in which there are precisely two relations that are not strongly regular of Latin square type or negative Latin square type.
It is clear that for $d=2$, there are zillions of such schemes. Also for $d=3$, there are infinitely many examples known, as we shall see in Section \ref{sec:bilinear}. The problem becomes interesting for $d \geq 4$. A first example is a $4$-class scheme related to the Sylvester graph, the unique distance-regular graph with intersection array $\{5,4,2;1,1,4\}$ and spectrum $\{5^1,2^{16},-1^{10},-3^9\}$.

Consider the generalized quadrangle GQ$(2,2)$. It has $6$ spreads (a set of lines such that each point is on precisely one line) and dually $6$ ovoids (a set of points such that each line contains exactly one point). It can be verified that two different spreads have precisely one line in common, and two different ovoids have precisely one point in common\footnote{The GQ$(2,2)$ has as points all pairs of a set of $6$ elements (Sylvester's duads) and as lines all triples of disjoint pairs (perfect matchings of $K_6$; Sylvester's synthemes). The ovoids are the sets of (5) pairs through a fixed element and the spreads are the $1$-factorizations of $K_6$.}. Let $\mathcal{S}$ be the set of spreads and $\mathcal{O}$ be the set of ovoids. 

The vertex set of the $4$-class scheme will be $\mathcal{S} \times \mathcal{O}$. 
The first relation is the Sylvester graph: $(S_1,O_1)$ is adjacent to $(S_2,O_2)$ if $S_1 \neq S_2$, $O_1 \neq O_2$, and the unique point in both $O_1$ and $O_2$ is on the unique line that is in both $S_1$ and $S_2$, see \cite[p.395]{BCN}.
Then we have two relations of Latin square type; both are disjoint unions of cliques: one is given by cliques $\{(S,O) : O \in \mathcal{O}\}$ and the other by  $\{(S,O) : S \in \mathcal{S}\}$.
These two relations together actually form the distance-$3$ graph of the Sylvester graph, and so the remaining relation is its distance-$2$ graph.

The thus obtained $4$-class scheme (indeed, it is) has first eigenmatrix 
$$P=\begin{bmatrix}
    1& 5& 20& 5& 5\\
    1& 2& -1 &-1 &-1 \\
    1& -3& 4 & -1& -1\\
    1& -1& -4& 5 & -1 \\
    1& -1& -4& -1& 5     
\end{bmatrix}.$$

Thus, this $4$-class association scheme is a fission scheme of the $3$-class scheme of the Sylvester graph. In this $3$-class association scheme, one of its relations is strongly regular of Latin square type, and we managed to fission this relation into multiple relations (of the same type). This is the key to almost all of our constructions.

\section{Combining common fission schemes}\label{sec:combining}

Our main tool is that if we have two fission schemes of a 2-class scheme that fission in such a way that also each of the distinct nontrivial idempotents fissions, then the common fission is also an association scheme.

\begin{lem}\label{lem:combine}
    Let $\{I,A,B\}$ be a $2$-class association scheme with nontrivial idempotents $E$ and $F$. Let $A_i$, $i=1,\dots,e$ and $E_i$, $i=1,\dots,e$ be such that $\{I,A_1,\dots,A_e,B\}$ is an association scheme with nontrivial idempotents $F$ and $E_i$, $i=1,\dots,e$, and let $B_j$, $j=1,\dots,f$ and $F_j$, $j=1,\dots,f$ be such that $\{I,B_1,\dots,B_f,A\}$ is an association scheme with nontrivial idempotents $E$ and $F_j$, $j=1,\dots,f$. Then $\{I,A_1,\dots,A_e,B_1,\dots,B_f\}$ is an association scheme with nontrivial idempotents $E_i$, $i=1,\dots,e$ and $F_j$, $j=1,\dots,f$.
\end{lem}

\begin{proof}
Let $\mathcal{A}=\langle I,A_1,\dots,A_e,B\rangle$ and $\mathcal{B}=\langle I,B_1,\dots,B_f,A\rangle$ be the Bose-Mesner algebras of the given fission schemes of $\{I,A,B\}$. Then $\mathcal{A}=\langle J,E_1,\dots,E_e,F\rangle$
and $\mathcal{B}=\langle J,F_1,\dots,F_f,E\rangle$.

We note that $A=A_1+\cdots+ A_e$, $B=B_1+\cdots+ B_f$, $E=E_1+\cdots+ E_e$, and $F=F_1+\cdots+ F_f$. From this, it follows that $A_i\circ B_j=A_i\circ A \circ B_j=0$ and $E_iF_j=E_iEF_j=0$ for $i=1,\dots,e$ and $j=1,\dots,f$.

Let $$\mathcal{C}_{AB}=\langle I,A_1,\dots,A_e,B_1,\dots,B_f\rangle$$ and 
$$\mathcal{C}_{EF}=\langle J,E_1,\dots,E_e,F_1,\dots,F_f\rangle.$$ 

From the above, it follows that $\mathcal{C}_{AB}$ is closed under entrywise multiplication and contains $I$, and that $\mathcal{C}_{EF}$ is closed under ordinary matrix multiplication and contains $J$.

Using the different bases of the Bose-Mesner algebras $\mathcal{A}$ and $\mathcal{B}$, it follows easily that $\mathcal{C}_{AB}=\mathcal{C}_{EF}$. Thus, by \cite[Thm.~2.6.1]{BCN}, $\mathcal{C}_{AB}$ is the Bose-Mesner algebra of an association scheme with nontrivial relations $A_1,\dots,A_e,B_1,\dots,B_f$ and nontrivial idempotents $E_1,\dots,E_e,F_1,\dots,F_f$. 
\end{proof}

We note that this lemma can be generalized easily to common fission schemes of larger-class schemes. For us, the most relevant consequence of the above lemma is the following. 

\begin{thm}\label{thm:main}
    Let $\{I,A_1,\dots,A_e,B\}$ be an association scheme on $n^2$ vertices with nontrivial idempotents $F$ and $E_i$, $i=1,\dots,e$, where $B$ is strongly regular of Latin square type or negative Latin square type with valency $k$, such that $F$ has multiplicity $k$ and $BF=(n-t)F$, where $t=k/(n-1)$.

    If $k \neq \frac12 (n^2-1)$ and $B$ can be decomposed into strongly regular graphs $B_1,\dots,B_f$ of the same type as $B$, then 
    $\{I,A_1,\dots,A_e,B_1,\dots,B_f\}$ is an association scheme.
\end{thm}

\begin{proof}
We first note that because $F$ has multiplicity $k$ for eigenvalue $n-t$ of $B$, it follows that $B$ has eigenvalue $-t$ for the idempotents $E_i$, for $i=1,\dots,e$. Thus, if we let $A=A_1+\cdots+A_e$ and $E=E_1+\cdots+E_e$, then $\{I,A,B\}$ is a $2$-class association scheme with nontrivial idempotents $E$ and $F$.

Since $A$ is the complement of $B$, it is a strongly regular graph of the same type as $B$. Thus, the complete graph can now be decomposed into strongly regular graphs $A$ and $B_1,\dots,B_f$, that are all of the same type. By \cite[Thm.~3]{D3}, we then obtain that $\{I,A,B_1,\dots,B_f\}$ is an amorphic association scheme. 
Next, let $t_i$ be such that $B_i$ has valency $k_i=t_i(n-1)$ for $i=1,\dots,f$. 
According to \cite[Prop.~1]{vDM2010}, the first eigenmatrix of the amorphic scheme can then be written as

$$P=\begin{bmatrix}
1& (n+1-t)(n-1)& t_1(n-1) &t_2(n-1)& \dots & t_f(n-1)\\
1& t-1& -t_1 &-t_2 & \dots & -t_f \\
1 &-(n+1-t) &n-t_1& -t_2 &\dots&-t_f\\
1 &-(n+1-t)& -t_1 &n-t_2 &\dots &-t_f\\
\vdots &\vdots &\vdots &\vdots &\ddots &\vdots\\
1& -(n+1-t) &-t_1& -t_2 &\dots &n-t_f 
\end{bmatrix}.$$

Now let $F_0,F_1,\dots,F_f$ be the nontrivial idempotents of this scheme, where $F_0$ is the one having eigenvalue $t-1$ for $A$. Then $F_0$ has multiplicity $n^2-1-k$ and $F_i$ has multiplicity $t_i(n-1)$ for $i=1,\dots,f$. Now it follows that the fusion scheme $\{I,A,B\}$ has nontrivial idempotents $F_0$ and $F_1+\cdots+F_f$.
Because $k \neq (n^2-1)/2$, it follows from the multiplicities that $E=F_0$ and $F=F_1+\cdots+F_f$.

Thus, we have all the conditions to apply Lemma \ref{lem:combine}, and conclude that 
$\{I,A_1,\dots,A_e,B_1,\dots,B_f\}$ is an association scheme.
\end{proof}

We recall that the condition that $F$ has multiplicity $k$ and $BF=(n-t)F$ is equivalent to the condition that the scheme idempotent $F$ is the same as the spectral ($2$-class) idempotent, because in a strongly regular graph of Latin square type or negative Latin square type, eigenvalue $n-t$ has multiplicity $k$. In Section \ref{sec:NLS}, we will argue that we need this condition.

We also observe that the condition $k\neq(n^2-1)/2$ is necessary because the $2$-class association scheme of the Paley graph on $81$ vertices (with valencies 40 and 40) has two $3$-class fission schemes whose common fission is not an association scheme. The first $3$-class scheme is a fusion of the $4$-class cyclotomic scheme, with relations having valencies 40, 20, and 20 (all are strongly regular of negative Latin square type). The second $3$-class scheme is a fusion of the $10$-class cyclotomic scheme, with relations having valencies 40, 32, and 8 (all are strongly regular of Latin square type). Their common fission is a so-called strongly regular decomposition of the complete graph into graphs with valencies 20, 20, 32, and 8, see \cite{D3}, but it is clearly not an association scheme because it has strongly regular relations of (strictly) different types.

\section{Constructions}\label{sec:constructions}

In Sections \ref{sec:bilinear}-\ref{sec:brouwer}, we discuss certain distance-regular graphs with diameter $3$ for which the distance-$3$ relation is strongly regular of Latin square type. They all have intersection array $\{q^3-1,q^3-q,q^3-q^2+1;1,q,q^2-1\}$. The corresponding association scheme has first eigenmatrix 

$$P=\begin{bmatrix}
    1& q^3-1& (q^2-1)(q^3-1)& (q^3-q^2+1)(q^3-1)\\
    1& q^2-1& q^3-2q^2+1 &-(q^3-q^2+1) \\
    1& -q^2-1& q^3+1& -(q^3-q^2+1)\\
    1& -1& -q^2+1 & q^2-1       
\end{bmatrix}$$
and the multiplicity of the idempotent corresponding to the bottom row, $F$ say, equals $(q^3-q^2+1)(q^3-1)$. Thus, Theorem \ref{thm:main} applies if we can fission the distance-$3$ relation into multiple strongly regular graphs of Latin square type.

\subsection{The symmetric bilinear forms graph}\label{sec:bilinear}

Let $q$ be an even prime power. Let $V$ be the vector space of $3 \times 3$ symmetric matrices over $GF(q)$ and let two matrices in $V$ be adjacent in $\Gamma$ whenever their difference has rank $1$. Then $\Gamma$ is distance-regular with diameter $3$ with the above intersection array. 
Two matrices are at distance $2$ if their difference has rank $2$ but is not alternating (i.e., it has no zero diagonal). Accordingly, $\Gamma_3$ splits naturally: two matrices are adjacent in $A_4$ (say) if their difference is alternating, i.e, it has rank $2$ and zero diagonal. Because the alternating matrices form a ($3$-dimensional) subspace of $V$, it follows that $A_4$ is a disjoint union of $q^3$-cliques (a square spread). Thus, we obtain a $4$-class scheme with precisely two relations that are strongly regular of Latin square type. See also the diagram in \cite[p.286]{BCN}.

\subsection{Brouwer-Pasechnik graphs}\label{sec:brouwer}
Let $q$ be a prime power. Brouwer and Pasechnik \cite[\S 3]{BP2011} constructed a distance-regular graph $Z$ with diameter $3$ on vertex set $W \times W$, where $W$ is the vector space $GF(q)^3$. Vertex
$(u,u')$ is adjacent to $(v,v')$ in $Z$ if $v'-u'=u \times v$, where $\times$ is an outer product on $W$.

Again, our aim is to fission the distance-$3$ graph $Z_3$ by obtaining multiple square spreads ($q^3$ mutually disjoint $q^3$-cliques) that share no edges in $Z_3$. Now it is not hard to show that $(w,w')$ is at distance $3$ from $(u,u')$ if $w=u$ (and $w' \neq u'$) or $w'-u'$ is not orthogonal to $w-u$.

It is then clear that we have at least one spread:
$$S_0=\{\{(w,w'): w' \in W\}: w \in W\}$$ (that is, the cliques with constant first coordinate).
Since we want different spreads not to share any edges, the other spreads will have cliques where all first coordinates will be distinct. Moreover, because in cliques we want that $w'-u'$ is not orthogonal to $w-u$, this also holds for the second coordinates.

Suppose now that we have one such clique $C$. It then contains vertices $(w,\varphi(w)), w \in W$, for some bijection $\varphi$.

Now it is relatively easy to show that $S_{\alpha}$, where
$$S_{\alpha}=\{\{(w,\alpha \varphi(w)+y):  w \in W \}: y \in W\}$$ 
is a spread for each $\alpha \in GF(q)^*$. Moreover, cliques in different spreads intersect in precisely one vertex, and hence share no edges. 
In total, we have $q$ square spreads, and hence, by repeatedly using Lemma \ref{lem:spread}, we can split $Z_3$ into $q+1$ new relations that are strongly regular of Latin square type. By fusion, we can thus obtain $d$-class association schemes with precisely $d-2$ relations that are of Latin square type, for $3 \leq d \leq q+3$. 

Thus, what remains is to find a clique $C$ in the graph, $Z_4$ say, obtained by removing the edges of $S_0$ from $Z_3$.

But the complement of $Z_4$ is a well-known graph, i.e., it is the affine polar graph $VO_6^+(q)$, see \cite[\S 3.3.1]{BvMSRG}, and a coclique of size $q^3$ in this graph is another well-known object, i.e., it corresponds to an ovoid in the hyperbolic quadric $Q^+(7,q)$, see \cite[(3)]{Ovoids}. We also refer to \cite{Ovoids,BKM11} for an overview of the known constructions of such ovoids. For us, it is relevant that these exist for infinitely many $q$. Thus, we have the following result.


\begin{thm}\label{thm:mainresult}
For all $d>1$, there exists a $d$-class association scheme with exactly $d-2$ relations of Latin square type.
\end{thm}

\section{Negative Latin square type}\label{sec:NLS}

So far, the (positive) constructions do not involve strongly regular graphs of negative Latin square type. 

\subsection{A fission scheme of a Hamming cube}

De Caen and Van Dam \cite{CD} constructed (among others) a $5$-class linear fission scheme of the Hamming scheme\footnote{The Doob scheme can be fissioned similarly} $H(3,4)$ by fissioning the distance-$3$ relation into three new relations that are each isomorphic to the Hamming graph itself. Since it is linear, it corresponds to a partition of $PG(2,4)$ into $5$ sets: these are $\{(001),(010),(100)\}$ (corresponding to the Hamming graph), $\{(xyz): xyz=\alpha\}$ for $\alpha \in GF(4)^*$ and $\{(xyz):x^3+y^3+z^3=0\}$. The latter is the Hermitian unital and comprises of the words of weight $2$. Its first eigenmatrix is
$$P=\begin{bmatrix}
1& 9& 9 &9& 9& 27\\
1& 5& -3 &-3 &-3 &3 \\
1 &-3 &5& -3 &-3 &3\\
1 &-3& -3 &5 &-3 &3\\
1 &-3 &-3 &-3 &5 &3\\
1& 1 &1& 1 &1 &-5 
\end{bmatrix}.$$

This scheme has several interesting fusion schemes. By fusing one pair of the relations with valencies 9, we end up with a $4$-class scheme with two relations that are strongly regular of negative Latin square type (and precisely two others that are not). Its first eigenmatrix is

$$P=\begin{bmatrix}
1& 9& 9 & 18& 27\\
1& 5& -3 &-6 &3 \\
1 &-3 &5& -6 &3\\
1 &-3& -3 &2 &3\\
1& 1 &1& 2 &-5 
\end{bmatrix}.$$

Note that by fusing also the other two relations with valency $9$, we obtain an amorphic $3$-class scheme. Since this scheme is linear, it corresponds to a partition of the projective plane of order $4$ into three $2$-intersecting sets; in this case two hyperovals (the union of $\{(001),(010),(100)\}$ and $\{(xyz): xyz=1\}$ is the classical hyperoval) and a unital (see \cite{CD}).

Next, we will argue that we need the condition on the idempotent $F$ in Theorem \ref{thm:main}.  

Consider one of the strongly regular relations $R$ of valency $18$ in the above amorphic $3$-class scheme. Since the Hamming graph $H(3,4)$ clearly admits a spread (a disjoint union of 16 $4$-cliques) and $R$ is the union of two Hamming graphs, also $R$ admits a spread (this corresponds in fact to a single point in $PG(2,4)$). Brouwer \cite{B3} observed that by removing a spread from a pseudo-geometric strongly regular graph, one obtains a distance-regular graph with diameter $3$, in this case with $c_2=4$, 
and hence a $3$-class scheme (see also \cite{HT}). In this case, it has first eigenmatrix

\begin{equation}\label{eq:P64}
P=\begin{bmatrix}
    1& 15& 45& 3\\
    1& 3& -3 &-1 \\
    1& -1& -3 & 3\\
    1& -5& 5& -1     
\end{bmatrix}
\end{equation}
and multiplicities $1,30,15,$ and $18$.

Even though we can split the strongly regular graph with valency $45$ into multiple negative Latin square type strongly regular graphs (after all, it is the union of two relations in the amorphic scheme), we cannot use this to fission the now obtained $3$-class scheme. Indeed, if that were possible, then the new valency $n_2$ and the new intersection number $p^{1}_{12}$ would need to satisfy the equation $15p^{1}_{12}=c_2n_2=4n_2$ (by double counting), which is impossible if $n_2$ is not a multiple of $15$. Note also that $45$ is not one of the multiplicities of the scheme.

It seems challenging to obtain larger-class association schemes with precisely 2 relations that are not of negative Latin square type. Note that for a strongly regular graph of negative Latin square type with $v=n^2$ and $k=t(n-1)$, it is required that $-t \gtrsim \sqrt{-n}$ (otherwise $\lambda <0$), so we cannot have more than about $\sqrt{-n}$ disjoint negative Latin square type graphs.

\subsection{A folded halved cube - some considerations}

The folded halved $12$-cube is distance-regular with diameter $3$ for which the distance-$2$ graph is strongly regular of negative Latin square type. The first eigenmatrix of the corresponding (self-dual) association scheme is 

$$P=\begin{bmatrix}
    1& 66& 495& 462\\
    1& 26& 15 &-42 \\
    1& 2& -17 & 14\\
    1& -6& 15& -10     
\end{bmatrix}.$$

The Shi-Krotov-Sol\'{e} graph \cite{SKS} is also distance-regular with diameter $3$ and its distance-$2$ graph is strongly regular of negative Latin square type with the same parameters.
We do not know however whether we can split the corresponding strongly regular graphs. In the following, we do construct strongly regular graphs with the same parameters that do split.

 
 
By a similar method as Polhill \cite{Polhill}, we can construct an amorphic $4$-class schemes with first eigenmatrix 
$$P=\begin{bmatrix}
    1& 231& 264& 264 & 264\\
    1& -25 & 8 &8 & 8 \\
    1& 7& -24 & 8 & 8\\
    1& 7& 8& -24 & 8 \\
    1 & 7 & 8 & 8 & -24
\end{bmatrix},$$
as a (nonstandard) product of two smaller amorphic schemes. First, consider the affine plane $AG(2,8)$. Each of the nine parallel classes gives a disjoint union of eight $8$-cliques, and together these form an amorphic $9$-class scheme. By fusing, we can obtain a $4$-class scheme with relations $B_1,B_2,B_3,B_4$ with respective valencies $21,14,14,$ and $14$.

Secondly, we consider the cyclotomic $3$-class scheme on $GF(16)$ in which all relations $C_1,C_2,C_3$ are Clebsch graphs.
Now, let
\begin{align*}
&A_1 = I \otimes B_1 + C_1 \otimes B_2 + C_2 \otimes B_3 + C_3 \otimes B_4,\\
&A_2 = I \otimes B_2 + C_1 \otimes (B_1 + I) + C_2 \otimes B_4 + C_3 \otimes B_3,\\
&A_3 = I \otimes B_3 + C_1 \otimes B_4 + C_2 \otimes (B_1 + I) + C_3 \otimes B_2,\\
&A_4 = I \otimes B_4 + C_1 \otimes B_3 + C_2 \otimes B_2 + C_3 \otimes (B_1 + I),    
\end{align*}
where $\otimes$ is the Kronecker product.
These four relations form an amorphic scheme. Moreover, if we let $C=C_1$, $B=B_1+B_2$, and $A=A_1+A_2$, then
$A=(I+C)\otimes(I+B)+\overline{C}\otimes \overline{B} -I$, which is a strongly regular graph of negative Latin square type with valency $495$ that can be decomposed (by construction).

\section{Different types of strongly regular graphs in an association scheme}\label{sec:types}

In \cite[p.188]{D3}, an infinite family of commuting decompositions of the complete graph into a strictly negative Latin square type graph, a strictly Latin square type graph, and a disjoint union of complete bipartite graphs was constructed. The smallest case contains a commuting Clebsch graph and Lattice graph. It also shows that the Clebsch graph has a square spread (and the left-over graph is the $2$-clique extension of the Cube). The corresponding ``eigenmatrix" (in the sense of \cite{D3}) equals
$$P=\begin{bmatrix}
    1& 5& 3& 7 \\
    1& 1 & 3 & -5 \\
    1& 1 & -1 & -1\\
    1& -3 & 3&  -1 \\
    1& -3 & -1 & 3
\end{bmatrix}.$$
As this matrix is non-square, the decomposition is not an association scheme.

\subsection{Positive and negative}

We next show that indeed there can be no association scheme in which there are both relations that are strongly regular of strictly Latin square type and relations that are strongly regular of strictly negative Latin square type. This was first proven using Bolzano \cite{Bolzano} by Václav Rozhoň, Robert Šámal, and Adrián Zámečník [private communication, June 2026]. Their key idea is to exploit the integrality of certain intersection numbers.

\begin{thm}
    There exists no association scheme having both a relation that is strongly regular of strictly Latin square type and a relation that is strongly regular of strictly negative Latin square type.
\end{thm}

\begin{proof}
    Suppose, to the contrary, that there is an association scheme on $n^2$ vertices with relation $A_1$ that is strongly regular with valency $k_1=t(n-1)$ and restricted eigenvalues $n-t$ and $-t$ and relation $A_2$ that is strongly regular with valency $k_2=w(n+1)$ and restricted eigenvalues $w$ and $w-n$, for some positive integers $n,t,$ and $w$, with $t\neq (n+1)/2$ and $w\neq (n-1)/2$.

    It follows from \cite[Lemma 2]{D3} that the two relations share $t(n-2w-1)$ linearly independent eigenvectors for their nonnegative restricted eigenvalues and $w(n-2t+1)$ linearly independent eigenvectors for their negative eigenvalues, implying that $t<(n+1)/2$ and $w<(n-1)/2$.

    Next, we consider certain intersection numbers. We first note that $p^2_{11}=\mu_1=t(t-1)$ and $p^1_{22}=\mu_2=w(w+1)$ (these are the standard parameters of the strongly regular graph that follow from the eigenvalues).

    Now it follows that $p^1_{12}=\frac{k_2}{k_1}p^2_{11}=\frac{w(n+1)(t-1)}{n-1}$ and $p^2_{21}=\frac{k_1}{k_2}p^1_{22}=\frac{t(n-1)(w+1)}{n+1}$.
    We will next use the integrality of these intersection numbers to derive a contradiction. To do this, we distinguish two cases according to the parity of $n$.

    First, let $n$ be even. Because $\gcd(n-1,n+1)=1$, it follows from the above expressions that $n-1 \mid w(t-1)$ and $n+1 \mid t(w+1)$. Let $u_1=w(t-1)/(n-1)$ and $u_2=t(w+1)/(n+1)$. Then it follows (after some straightforward calculations) that $u_2-u_1>0$, while also $u_2-u_1<1$ (using that $t+w \leq n-1$), thus giving a contradiction.

    Secondly, let $n$ be odd, say $n=2m+1$. It now follows that $m\mid w(t-1)$ and $m+1 \mid t(w+1)$. Let $u_1=w(t-1)/m$ and $u_2=t(w+1)/(m+1)$. Then it follows (again, after some straightforward calculations) that $0<u_2-u_1<2$, which implies that $u_2-u_1=1$. This however leads to the equation $(m-w)(m-t+1)=0$, which shows that either $w=(n-1)/2$ or $t=(n+1)/2$, which we excluded from the beginning. Thus we arrive at the final contradiction.    
\end{proof}




More generally, we may wonder which (negative) Latin square type graphs can coexist with what (other) kinds of strongly regular graphs in an association scheme.

\subsection{Square spreads}
Let us start to consider the smallest Latin square type graph: a square spread (that is, a disjoint union of $n$ cliques of order $n$). So, we assume that we have an association scheme in which one of the relations is a square spread. Then the scheme is imprimitive, and each other relation in the scheme has a quotient relation; and the distinct quotient relations form the quotient scheme \cite{BBIT}. Indeed, let $P$ be a matrix with columns the characteristic vectors of the cliques of the spread. If $A$ is another relation in the scheme, then there is a $b \in \mathbb{N}$ such that the quotient matrix $B=\frac{1}{n}P^{\top}AP$ equals $b \widetilde{A}$ for some $01$-matrix $\widetilde{A}$. It is well known (and easy to see) that each eigenvalue of $B$ is also an eigenvalue of $A$.

Next, we consider the case that $A$ is strongly regular. Then it follows that $\widetilde{A}$ must be complete or strongly regular. If it is a complete graph, then the quotient scheme is complete as well (i.e., it is a $1$-class scheme). We then obtain the following.

\begin{prop} Let $\mathbb{A}$ be an association scheme with a square spread such that the corresponding quotient scheme is complete. If a relation of $\mathbb{A}$ is strongly regular, then it is of Latin square type.
\end{prop}

\begin{proof}
    The quotient $\widetilde{A}$ of $A$ is a complete graph, so it has valency $n-1$ and another eigenvalue $-1$. Thus, the strongly regular graph $A$ has valency $k=b(n-1)$ and an eigenvalue $-b$. By Lemma \ref{lem:SRGkandaNLS} (note that in this case $n$ is positive), it follows that $A$ is of Latin square type. 
\end{proof}




If the quotient scheme is not complete (hence $\widetilde{A}$ is strongly regular) and $b$ takes one of the two extremal values ($1$ or $n$), then we have the following. First of all, if $b=n$, then $A$ is the coclique extension of $\widetilde{A}$, that is, $A=\widetilde{A} \otimes J_n$, which implies that it has an eigenvalue $0$, and hence that $A$ is a complete multipartite graph.
Secondly, if $b=1$, then the distinct eigenvalues of $A$ and its quotient are the same, and hence so are their parameters $k$, $\lambda$, and $\mu$. But the number of vertices $v$ is not the same, and hence the (well-known) equation $k(k-1-\lambda)=\mu(v-1-k)$ implies that $\mu=0$, and hence $A$ is a disjoint union of cliques.

Both cases occur in the following example.
The wreath product of a $1$-class scheme on $n/2$ vertices and the $3$-class distance scheme of $K_{n,n}$ minus a matching (i.e., it has $3$ relations that are $n/2$ copies of the relations of the $3$-class scheme and a complete $n/2$-partite relation between the copies) gives an infinite family of $4$-class schemes with first eigenmatrix
$$P=\begin{bmatrix}
    1& (n/2-1)2n& n-1& n-1 & 1\\
    1& -2n & n-1 & n-1 & 1 \\
    1& 0 & 1 & -1 & -1\\
    1& 0 & -1&  -1& 1 \\
    1& 0 & 1-n & n-1 & -1
\end{bmatrix}.$$
Indeed, this $4$-class scheme contains a square spread, a complete multipartite graph, and a disjoint union of cliques (a matching).

\subsection{The lattice graph}

Next we consider the lattice graph. Here the obtained result is stronger than for the square spread. We first consider one of the idempotents in the scheme.

\begin{lem}\label{lem:oneE} Let $\mathbb{A}$ be an association scheme in which one of the relations is the lattice graph $L_2(n)$. Then there is only one idempotent of the scheme on which this lattice graph has eigenvalue $n-2$.    
\end{lem}

\begin{proof} We let the vertices of the scheme be represented as a lattice: $V=\{(i,j): i,j=1,2,\dots,n\}$, where adjacency of the given lattice graph $L_2(n)$ is given by matching first or second coordinate (as usual), and let $A$ be the corresponding adjacency matrix. It is well-known that $A$ has (spectral, or $2$-class) idempotent $E$ for eigenvalue $n-2$ given by $$n^2E=(2n-2)I+(n-2)A-2(J-I-A).$$ More specifically, each eigenvector is the sum of two vectors, where one is constant on the horizontal lines of the lattice and orthogonal to each of the vertical lines\footnote{The sum of entries on each vertical line is zero}, and for the second it is just the other way around. To put it differently, for each eigenvector $f$, there are vectors $h$ and $v$ orthogonal to the all-ones vector such that $f_{(i,j)}=h_i+v_j$ for $i,j=1,2,\dots,n$.

    Now let $F$ be an idempotent of $\mathbb{A}$ of rank $m$ for which the lattice graph has eigenvalue $n-2$, and consider its column $f$ for vertex $(1,1)$. Since this is an eigenvector for $n-2$, we have that $f_{(i,j)}=h_i+v_j$, where $h$ and $v$ are orthogonal to the all-ones vector. Because $F$ has constant diagonal and trace $m$, it follows that $h_1+v_1=m/n^2$. Moreover, $n^2F \circ A= \theta A$ for some dual eigenvalue $\theta$. By considering the horizontal neighbors of $(1,1)$, this implies that $v_2=v_3=\cdots=v_n$, and by orthogonality to the all-ones vector, it follows that $v_j=-v_1/(n-1)$ for $j \neq 1$. Similarly, it follows that $h_i=-h_1/(n-1)$ for $i \neq 1$, and hence $f_{(i,j)}=h_i+v_j=-(h_1+v_1)/(n-1)=-m/(n^2(n-1))$ for $i \neq 1$ and $j \neq 1$. In particular, $f$ is constant on the non-neighbors of $(1,1)$. As this holds analogously for all vertices and corresponding columns of $F$ (with the same constants), it follows that $n^2F=mI +\theta A - \frac{m}{n-1}(J-I-A)$. From the fact that $F$ has row sums zero, it now follows that $\theta=\frac{m(n-2)}{2(n-1)}$, hence $F$ is a scalar multiple of $E$, and so $F=E$.
\end{proof}

Let us now look at bit closer (and more generally) into the eigenspace for eigenvalue $n-2$.
Let $A$ be a strongly regular graph of Latin square type or negative Latin square type, say with valency $t(n-1)$ and restricted eigenvalues $n-t$ and $-t$, and $B$ be a strongly regular graph with valency $k$ and restricted eigenvalues $r$ and $s$, with $r-s$ having the same sign as $n$. If $A$ and $B$ commute, then they share eigenvectors for eigenvalues $n-t$ and $s$ \cite[Lemma 2]{D3}. Moreover, it follows from \cite[Lemma 2]{D3} that the dimension of the common eigenspace for eigenvalues $n-t$ and $r$ equals $-t\frac{k+(n-1)s}{r-s}$. This implies that $A$ and $B$ share an eigenvector for eigenvalues $n-t$ and $r$ unless $k=-s(n-1)$, and by Lemma \ref{lem:SRGkandaNLS} this exception is exactly the case when $B$ is of the same type as $A$.
Thus, the above lemma implies the following.

\begin{prop} Let $\mathbb{A}$ be an association scheme in which one of the relations is a lattice graph. If a relation of $\mathbb{A}$ is strongly regular, then it is of Latin square type.
\end{prop}

\begin{proof}
    By the above considerations, if $A$ is the lattice graph and $B$ a strongly regular graph that is not of Latin square type, then there are at least two idempotents that have eigenvalue $n-2$ for $A$, that is, at least one for each of the distinct restricted eigenvalues of $B$. This contradicts Lemma \ref{lem:oneE}. 
\end{proof}

We suspect that this proposition holds more generally by replacing the lattice graph by Latin square graphs, but we have not been able to prove this.

We finally note that the association scheme with eigenmatrix \eqref{eq:P64} has a negative Latin square type graph as one of its relations, and the idempotent for eigenvalue $n-t=-3$ indeed splits for the restricted eigenvalues of another strongly regular relation (not of negative Latin square type).

\section*{Acknowledgement}  
Jack Koolen is supported by the National Natural Science Foundation of China (No. 12471335). 
Yanzhen Xiong is supported by the National Natural Science Foundation of China (No. 12501502) and Innovation Research Foundation of College of Science at National University of Defense Technology (202501-YJRC-LXY-01).

\end{document}